\documentclass[12pt,a4paper]{article}
\usepackage{amsmath,amssymb,amsfonts}
\usepackage{amsthm}
\usepackage{mathtools}
\usepackage{authblk}
\usepackage{enumerate}
\usepackage[english]{babel}
\usepackage{enumitem}
\usepackage{todonotes}
\usepackage{indentfirst}
\usepackage{geometry}
\usepackage{graphicx}
\usepackage{tabularx}
\usepackage{tocbibind}
\usepackage{tikz}
\usepackage{array}
\usepackage{caption}
\usepackage{hyperref}
\usepackage{lipsum}
\usepackage{listings}
\usepackage{multicol,multirow}
\usepackage{xcolor,xspace}
\usepackage[english]{babel}
\usepackage[export]{adjustbox}
\newtheorem{theorem}{Theorem}[section]
\newtheorem{corollary}{Corollary}[theorem]
\newtheorem{lemma}[theorem]{Lemma}
\newtheorem{prop}[theorem]{Proposition}
\newtheorem{example}[theorem]{Example}
\theoremstyle{definition}

\theoremstyle{remark}

\title{\textbf{On Spectrum of Neighbourhood Corona Product of Signed Graphs}}
\author[1]{Bishal Sonar\thanks{Email: bsonarnits@gmail.com}}
\author[2]{Satyam Guragain\thanks{Email: shatym17@gmail.com}}
\author[3]{Ravi Srivastava\thanks{Corresponding author Email: ravi@nitsikkim.ac.in}}
\affil[1,2,3]{Department of Mathematics, National Institute of Technology Sikkim, South Sikkim 737139, India}
\date{}
\begin{document}
\parskip1ex
\parindent0pt
\maketitle

\begin{abstract}
    Given two signed graphs $\Gamma_1$ with nodes $\{u_1,u_2,\cdots,u_n\}$ and $\Gamma_2$, the neighbourhood corona, $\Gamma_1*\Gamma_2$ is the signed graph obtained by taking one copy of $\Gamma_1$ and $n_1$ copies of $\Gamma_2$, and joining every neighbour of the $i^{th}$ node with each nodes of the $i^{th}$ copy of $\Gamma_2$ by a new signed edge. In this paper we will determine the condition for $\Gamma_1*\Gamma_2$ to be balanced. We also determine the adjacency spectrum of $\Gamma_1*\Gamma_2$ for arbitrary $\Gamma_1$ and $\Gamma_2$, and Laplacian and signless Laplacian spectrum of $\Gamma_1*\Gamma_2$ for regular $\Gamma_1$ and arbitrary $\Gamma_2$, in terms of the corresponding spectrum of $\Gamma_1$ and $\Gamma_2$.
\end{abstract}

\section{Introduction}

    \noindent The paper examines graphs that are finite, undirected, and simple. Let $H=(V(H),E(H))$ be a graph with vertex set $V(H) = \{u_1,u_2,...,u_n\}$ and edge set $E(G)=\{e_1,e_2,..,e_m\}$. A signed graph $\Gamma(H,\sigma)$ is a tuple consisting of an unsigned graph $H=(V,E)$ and a mapping $\sigma:E(H)\rightarrow\{+,-\}$ known as the signature of $\Gamma$ which assign either positive sign or negative sign to edges. A cycle(C) is positive if $\prod_{e\in C} \sigma(e)=+$, other wise it is negative. A graph is balanced if all its cycles are positive~\cite{harary1953notion}. The signed degree of a vertex $v$ (denoted by $sdeg(u)$) equals $d^+_u-d^-_u$ where $d^+_u$ and $d^-_u$ denotes the positive and negative degrees of a vertex $u$ respectively and the degree of a vertex $u$ is given by $d_v=d^+_u+d^-_u$. The adjacency matrix of $\Gamma=(H,\sigma)$ is the $n\times n$ matrix given by $A(\Gamma)=(a^\sigma _{ij} )$ where $a^\sigma_{ij}=\sigma(u_iu_j)a_{ij}$ and $a_{ij}$ = 1 if $v_i$ and $u_j$ are adjacent and $a_{ij}=0$ otherwise. The vertex-edge incidence matrix of underlying graph $H$ of $\Gamma$ is the $n\times m$ matrix given by $R(H)=(b_{ij})$ where 
    \begin{equation*}
        b_{ij}=\begin{cases}
        1 & \text{if } u_i \text{ is an endpoint of edge } e_j\\ 
        0 & \text{otherwise}
    \end{cases}
    \end{equation*} 
    
    \noindent The signed Laplacian matrix of $\Gamma$ is given by $L(\Gamma)=D(\Gamma)-A(\Gamma)$ whereas the signless Laplacian matrix of $\Gamma$ is given by $Q(\Gamma)=D(\Gamma)+A(\Gamma)$, where $D(\Gamma)=diag(d(u_1),d(u_2),\cdots,d(u_n))$. It is clear that $R(H)I_mR(H)^T=Q(H)$ where $I_m$ is an identity matrix of order $m$.

    Given any $n\times n$ matrix $M$, denote characteristics polynomial of $M$ by, $f(M,x)=det(xI-M)$, where $I$ is the identity matrix of order same as that of $M$. We denote the eigenvalues of $A(\Gamma)$, $L(\Gamma)$ and $Q(\Gamma)$, respectively, by \[\lambda_1(\Gamma)\leq \lambda_1(\Gamma)\leq...\leq\lambda_n(\Gamma),\]
    \[\gamma_1(\Gamma)\leq\gamma_2(\Gamma)\leq...\leq\gamma_n(\Gamma),\]
    \[\nu_1(\Gamma)\leq\nu_2(\Gamma)\leq...\leq\nu_n(\Gamma).\]
    The collection of eigenvalues of $A(\Gamma), L(\Gamma), and Q(\Gamma)$ is denoted by $S(A(\Gamma)), S(L(\Gamma))$, and $S(Q(\Gamma))$ respectively. Two graphs are called $M$-co-spectral if they have the same $M$-spectrum, $M\in\{A(\Gamma),L(\Gamma),Q(\Gamma)\}.$ 
    
    A signed graph $\Gamma=(H,\sigma)$ is net-regular with net degree $k$ ($k$ being an integer) if the signed degree of all the vertices in $\Gamma$ equals $k$~\cite{nayak2016net} whereas $\Gamma$ is said to be co-regular, if the underlying graph $H$ is regular for some integer $r$ and $\Gamma$ is net-regular with net degree $k$. The pair $(r,k)$ is the co-regularity pair of $\Gamma=(H,\sigma)$~\cite{shahul2015co}.
    A marking is a function $\mu:V(G)\rightarrow\{+,-\}$ which assigns either a positive sign or negative sign to the vertices of a signed graph $\Gamma=(H,\sigma)$. Thus a signed graph is a 3-tuple $\Gamma=(H,\sigma,\mu)$. In this paper, we will mostly discuss canonical marking (denoted as $\mu$) although multiple marking function for a signed graph exists. In canonical marking $\mu(u)=\prod_{e\in E_u} \sigma(e)$ where $E_u$ is the set of signed edges adjacent to $u$. For further information about marking function refer to~\cite{adhikari2023corona} and references therein.\\
   
    Indulal Gopalapillai~\cite{gopalapillai2011spectrum} introduced the spectrum of neighbourhood corona of unsigned graphs and later in 2013 Liu and Zhou~\cite{liu2014spectra} extended the work using the concept of coronal of unsigned graphs. We are extending this work for signed graphs. The structural properties of the corona of two signed graphs were first introduced by Adhikari et. al. in ~\cite{adhikari2023corona}. We will use those concepts to study the structural properties of \textit{neighbourhood corona product} of two signed graphs.

    \begin{lemma}\label{lemma 0}
        (Schur complement)~\cite{bapat2010graphs} Let $A$ be a $n\times n$ matrix partitioned as \[\begin{bmatrix}
            A_{11}&A_{12}\\
            A_{21}&A_{22},
        \end{bmatrix}\]
        where $A_{11}$ and $A_{22}$ are square matrices. If $A_{11}$ and $A_{22}$ are invertible , then
        \begin{equation*}
        \begin{split}
            \det\begin{bmatrix}
                A_{11}&A_{12}\\
                A_{21}&A_{22},
            \end{bmatrix}
            &=\det(A_{22})\det(A_{11}-A_{12}A_{22}^{-1}A_{21})\\
            &=\det(A_{11})\det(A_{22}-A_{21}A_{11}^{-1}A_{12})
        \end{split}
        \end{equation*}
    \end{lemma}

    \noindent \textbf{Notation Used:}\\
    $a^{(n)}$: $a$ is repeated $n$ times. 
    
\section{Neighbourhood Corona product of signed graph}

    Let $\Gamma_1=(H_1,\sigma_1,\mu_1)$ and $\Gamma_2=(H_2,\sigma_2,\mu_2)$ be two signed graphs on $n_1$ nodes, $m_1$ edges and $n_2$ nodes and $m_2$ edges respectively. The neighbourhood corona product $\Gamma_1*\Gamma_2$ of $\Gamma_1$ and $\Gamma_2$ is a signed graph obtained by taking one copy of $\Gamma_1$ and $n_1$ copies of $\Gamma_2$, and joining every neighbour of the $i^{th}$ nodes of $\Gamma_1$ with each node of the $i^{th}$ copy of $\Gamma_2$ by a new signed edge, where the sign of the new edge between nodes $u$ which is the neighbour of $i^{th}$ nodes of $\Gamma_1$ and $j^{th}$ nodes $v$ in the $i^{th}$ copy of $\Gamma_2$ is given by $\mu_1(u)\mu_2(v)$, where $\mu_i$ is a marking defined by $\sigma_1$, and $\sigma_2$.

    If $V(H_1)=\{u_1,u_2,...,u_{n_1}\}$ and $V(H_2)=\{v_1,v_2,...,v_{n_2}\}$ then the marking vector of $\Gamma_1$ and $\Gamma_2$ are column vectors of dimension $n_1\times1$ and $n_2\times1$ given by $\mu(\Gamma_1)=[\mu_1(u_1),\mu_1(u_2),...,\mu_1(u_{n_1)}]^T$ and $\mu(\Gamma_2)=[\mu_2(v_1),\mu_2(v_2),...,\mu_2(v_{n_2)}]^T$ respectively, where $A^T$ denotes transpose of $A$ and for $i=1,2.$ \[\mu_i(w)=\begin{cases}
        +1 & \text{if marking of $w$ is +}\\
        -1 & \text{if marking of $w$ is $-$}
    \end{cases}\]
   Also we consider $\phi(\Gamma_j)=diag[\mu(\Gamma_j)]^T$ so that $\phi(\Gamma_j)^2=I_{n_j}$ for $i=1,2.$ Note that the neighbourhood corona product of $\Gamma_1*\Gamma_2$ has $n_1(n_2+1)$ nodes and $m_1+n_1m_2+2m_1n_2$ edges.

\subsection{Structural property of Neighbourhood corona}
    \noindent Structural property of corona of signed graphs has already 
being established by Adhikari at. el.~\cite{adhikari2023corona}, following that we present statistical information concerning the counts of nodes, edges, and traits in $\Gamma=\Gamma_1*\Gamma_2=(G,\sigma,\mu)$ for signed graphs $\Gamma_1=(G_1,\sigma_1,\mu_1)$ and $\Gamma_2=(G_2,\sigma_2,\mu_2)$, where $G=(V,E)$, $G_1=(V_1,E_1)$ and $G_2=(V_2,E_2)$. Clearly, the number of nodes in $\Gamma_1*\Gamma_2$ is $|V|=|V_1|(1+|V_2|)$. We define for $i=1,2$,
    \begin{align*}
		N_i^+	&= \text{count of nodes marked positive in $\Gamma_i$}\\
		N_i^-	&= \text{count of nodes marked negative in $\Gamma_i$}\\
		B^+_i	&= \text{count of nodes marked positive times its degree in $\Gamma_i$}\\
		B^-_i	&= \text{count of nodes marked negative times its degree in $\Gamma_i$}
	\end{align*}
 
    Then the table below describes the statistics of edges in $\Gamma_1*\Gamma_2$. \\
    \begin{table}[ht]
        \centering
        \begin{tabular}{|c|c|c|c|}
        \hline
        Edges & $\Gamma_1$  & $\Gamma_2$  & $\Gamma_1*\Gamma_2$ \\ \hline\hline
        Count of edges & $|E_1|$ & $|E_2|$ & $|E_1|+|V_1||E_2|+2|V_2||E_1|$ \\ \hline
        Count of positive edges & $|E_1^+|$ & $|E_2^+|$ &$|E_1^+|+|V_1||E_2^+|+B_1^+N_2^+ +B^-_1N_2^-$\\ \hline
        Count of negative edges & $|E_1^-|$ & $|E_2^-|$ &$|E_1^-|+|V_1||E_2^-|+B_1^+N_2^- +B^-_1N_2^+$\\ \hline
        \end{tabular}
        \caption{Count of edges in $\Gamma_1*\Gamma_2$}
        \label{T1}
    \end{table}
    \noindent Next, we will count the number of triads in $\Gamma_1*\Gamma_2$. Let $p=+,-$ and denote for $i=1,2$,\\
    $|E^p_i|^{\overset{+}{+}}=$ count of edges with sign $p$ which connects a pair of positive nodes in $\Gamma_i$\\
    $|E^p_i|^{\overset{+}{-}}=$ count of edges of sign $p$ which connects a pair opposite sign nodes\\
    $|E^p_i|^{\overset{-}{-}}=$ count of edges with sign $p$ which connects a pair of negatively marked nodes in $\Gamma_i$\\
    $|E^p_i|^{\overset{p}{p}}_c=$ count of edges with sign $p$ which connects a pair of nodes of sign $p$ and $p$ times its common neighbour nodes in $\Gamma_i$\\
    \begin{table}[ht]
        \centering
        \begin{tabular}{|c|c|c|c|} \hline 
            Triads &$\Gamma_1$ &$\Gamma_2$ &$\Gamma_1*\Gamma_2$\\ \hline\hline
            Count of $T_0$ &$|T_0(\Gamma_1)|$ &$|T_0(\Gamma_2)|$ &$|T_0(\Gamma_1)|+|V_1||T_0(\Gamma_2)|+B_1^+|E_2^+|^{\overset{+}{+}}+B_1^-|E_2^+|^{\overset{-}{-}}$  \\ &&&$+N_2^+|E_1^+|_c^{\overset{+}{+}}+N_2^-|E_1^+|_c^{\overset{-}{-}}$ \\ \hline
            Count of $T_1$ &$T_1(\Gamma_1)$ &$T_1(\Gamma_2)$ &$|T_1(\Gamma_1)|+|V_1||T_1(\Gamma_2)|+B_1^+(|E_2^+|^{\overset{+}{-}}+|E_2^-|^{\overset{+}{+}})+B_1^-(|E_2^+|^{\overset{+}{-}}$\\ &&&$+|E_2^-|^{\overset{-}{-}})+N_2^+(|E_1^+|_c^{\overset{+}{-}}+|E_1^-|_c^{\overset{+}{+}})+N_2^-(|E_1^+|_c^{\overset{+}{-}}+|E_1^-|_c^{\overset{-}{-}})$\\ \hline
            Count of $T_2$ &$|T_2(\Gamma_1)|$ &$|T_2(\Gamma_2)|$ &$|T_2(\Gamma_1)|+|V_1||T_2(\Gamma_2)|+B_1^+(|E_2^+|^{\overset{-}{-}}+|E_2^-|^{\overset{+}{-}})+B_1^-(|E_2^+|^{\overset{+}{+}}$\\ &&&$+|E_2^-|^{\overset{+}{-}})+N_2^+(|E_1^+|_c^{\overset{-}{-}}+|E_1^-|_c^{\overset{+}{-}})+N_2^-(|E_1^+|_c^{\overset{+}{+}}+|E_1^-|_c^{\overset{+}{-}})$\\ \hline
            Count of $T_3$ &$|T_3(\Gamma_1)|$ &$|T_3(\Gamma_2)|$  & $|T_3(\Gamma_1)|+|V_1||T_3(\Gamma_2)|+B_1^+|E_2^-|^{\overset{-}{-}}+B_1^-|E_2^-|^{\overset{+}{+}}+N_2^+|E_1^-|_c^{\overset{-}{-}}$\\ 
            &&&$+N_2^-|E_1^-|_c^{\overset{+}{+}}$ \\ \hline
            \end{tabular}
        \caption{Count of Triads in $\Gamma_1*\Gamma_2$}
        \label{T2}
    \end{table}
    $T_i=$ triad with $i$ negative edges\\

    \noindent The proofs of the formulas presented in Table \ref{T1} and Table \ref{T2} directly follow from the definition of neighbourhood corona product, and easily can be verified. Total triads of $\Gamma_1*\Gamma_2$ is given by \[T(\Gamma)=T(\Gamma_1)+|V_1|T(\Gamma_2)+|V_1|_c|E_2|+|V_2||E_1|_c\]
    where $T(\Gamma_i)$, i=1,2, denotes the total triads in $\Gamma_i$.\\
    It is very  much clear that $\Gamma_1*\Gamma_2$ is unbalanced if one of $\Gamma_1$ or $\Gamma_2$ is unbalanced. However, if both $\Gamma_1$ and $\Gamma_2$ are balanced then it is not certain that $\Gamma_1*\Gamma_2$ is balanced. For example, in Figure \ref{f1} the graph $\Gamma_1*\Gamma_2$ is unbalanced even if both $\Gamma_1$ and $\Gamma_2$ are balanced.\\
    \begin{figure}[ht]
        \centering
        \includegraphics[scale=0.6]{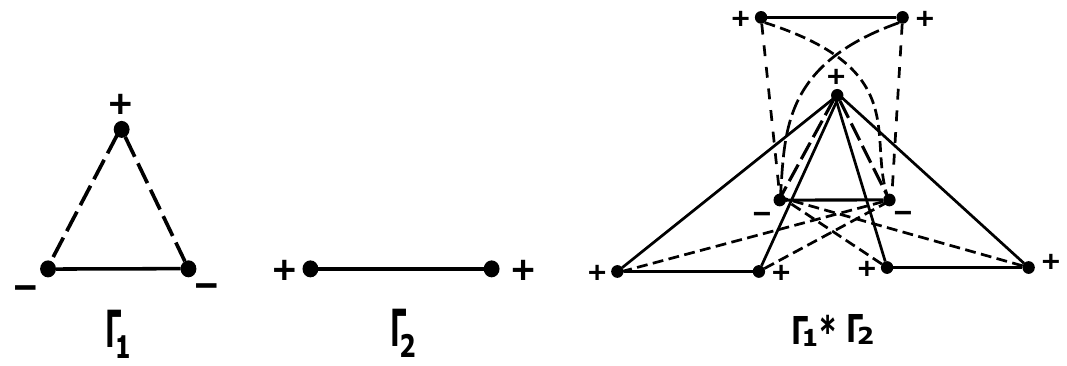}
        \caption{Neighbourhood Corona product of $\Gamma_1$ and $\Gamma_2$. The dotted line implies a negative edge and the normal line implies a positive edge. }
        \label{f1}
    \end{figure}
    
    \noindent The upcoming theorem will establish the criteria under which the signed graph $\Gamma_1*\Gamma_2$ is considered unbalanced.
    \begin{theorem}
        Let $\Gamma_1=(G_1,\sigma_1,\mu_1)$ and $\Gamma_2=(G_2,\sigma_2,\mu_2)$ be two balanced signed graphs. Then $\Gamma_1*\Gamma_2$ is unbalanced iff either $\Gamma_1$ or $\Gamma_2$ contains the following types of edges.
        \begin{enumerate}[label=(\roman*)]
            \item positive edge connecting opposite nodes
            \item negative edge connecting a pair of positive nodes
            \item negative edge connecting a pair of negative nodes.
        \end{enumerate}
    \end{theorem}
    \begin{proof}
        The proof is the consequence of the observation that any positive or negative marked nodes of $\Gamma_i$, $i=1,2,$ form a triad $T_1$ in $\Gamma_1*\Gamma_2$ when there are edges of type $(1)$ and/or $(ii)$ in $\Gamma_i$, $i=1,2,$. Otherwise, it results in the formation of a $T_3$ triad when $\Gamma_1$ or $\Gamma_2$ includes an edge of type $(iii).$
    \end{proof}
    \noindent In the next section we will study about the spectral properties of neighbourhood corona product.

\subsection{Spectra of neighbourhood corona}

    \noindent Cam McLeman, and Erin McNicholas~\cite{mcleman2011spectra} first introduced the coronal of graphs. Then Shu and Gui~\cite{cui2012spectrum} generalised it and defined coronal for Laplacian and signless Laplacian matrix of unsigned graphs. Recently Amrik at. el.~\cite{singh2023structural} defined coronal of signed graph. Given a signed graph $\Gamma=(G,\sigma,\mu)$ of order $n$ and $M$ be a graph matrix of $\Gamma$. Now, viewed as a matrix over the rational function field $\mathbb{C}(\lambda)$, the characteristic matrix $\lambda I_n-M$ has determinant non-zero, so it is invertible. The signed $M$-coronal $\chi_M(\lambda) \in \mathbb{C}(\lambda)$ of $\Gamma$ is defined as, 
    \begin{equation}
    \label{eqn1}
        \chi_M(\lambda)=\mu(\Gamma)^T(\lambda I_n-M)^{-1}\mu(\Gamma)
    \end{equation}
    If we replace $M$ in (\ref{eqn1}) by $A(\Gamma)$, $L(\Gamma)$, or $Q(\Gamma)$ then we will obtain signed $A(\Gamma)$-coronal, signed $L(\Gamma)$-coronal, or signed $Q(\Gamma)$-coronal of $\Gamma$ respectively.\\
    
    \noindent \textbf{Properties of \textit{Kronecker product}:}~\\
    Let $A=(a_{ij})_{n_1\times m_1}$ and $B=(b_{ij})_{n_2\times m_2}$, then the Kronecker product $A\otimes B$ of matrix $A$ and $B$ is a $n_1n_2\times m_1m_2$ matrix formed by replacing each $a_{ij}$ by $a_{ij}B$.
    \begin{enumerate}[label=(\roman*)]
        \item Kronecker product is associative.
        \item $(A\otimes B)^T=A^T\otimes B^T$.
        \item $(A\otimes B)(C \otimes D)=AC\otimes BD$, given the product $AC$ and $BD$ exists.
        \item $(A\otimes B)^{-1}=A^{-1}\otimes B^{-1}$, for non-singular matrices $A$ and $B$.
        \item If $A$ and $B$ are $n\times n$ and $m\times m$ matrices, then det$(A\times B)=(det A)^m(det B)^n$.
    \end{enumerate}

    Let $\Gamma_1=(G_1,\sigma_1,\mu_1)$ and $\Gamma_2=(G_2,\sigma_2,\mu_2)$ be arbitrary signed graphs on $n_1$ and $n_2$ vertices, respectively. Following~\cite{gopalapillai2011spectrum}, we first label the nodes of $\Gamma_1*\Gamma_2$ as follows. Let $V(\Gamma_1)=\{u_1,u_2,...,u_{n_1}\}$ and $V(\Gamma_2)=\{v_1,v_2,...,v_{n_2}\}$. For $i=1,2,...,n_1,$ let $v_1 ^i,v_2 ^i,...,v_{n_2} ^i$ denotes the nodes of the $i^{th}$ copy of $\Gamma_2$, with the understanding that $v_j^i$ is the copy of $v_j$ for each $j$. Denote
    \begin{equation}
    \label{eqn2}
                W_j=\big\{v_j^1,v_j^2,...,v_j^{n_1}\big\}, ~~~j=1,2,...,n_2
    \end{equation}
    Then $V(\Gamma_1)\cup W_1\cup W_2\cup...\cup W_{n_2}$ is a partition of $V(\Gamma_1*\Gamma_2)$. It is clear that the degrees of the nodes of $\Gamma_1*\Gamma_2$ are: 
    \begin{equation*}
        d_{\Gamma_1*\Gamma_2}(u_i)=(n_2+1)d_{\Gamma_1}(u_i), ~~~i=1,2,...,n_1\\
    \end{equation*}
    \begin{equation*}
        d_{\Gamma_1*\Gamma_2}(v^i_j)=d_{\Gamma_2}(v_j)+d_{\Gamma_1}(u_i), ~~~i=1,2,...,n_1, ~~~j=1,2,...,n_2.
    \end{equation*}

\subsubsection{A-spectra of signed graph}

     \begin{lemma}\label{Lemma1}
        Let $\Gamma=(G,\sigma,\mu)$ be $k$-net-regular signed graph on $n$ nodes, then signed $A(\Gamma)-$coronal of $\Gamma$ is, \[\chi_{A(\Gamma)}(\lambda)=\frac{n}{\lambda-k}\]
    \end{lemma}
    \begin{proof}
        Here $\mu(\Gamma)$ is either ${\mathbf{1_n}}$ or ${\mathbf{-1_n}}$ or $\mathbf{0_n}$ according as $r$ is positive or negative or zero where $\mathbf{1_n}$ and $\mathbf{0_n}$ denotes all $1-$column vector and all $0-$column vector of order $n$ respectively.\\
        Thus, $A(\Gamma)\mu(\Gamma)=k\mu(\Gamma)$, and hence
        \begin{equation*}
        \begin{split}
            \big(\lambda I_n-A(\Gamma)\big)\mu(\Gamma)&=\lambda\mu(\Gamma)-A(\Gamma)\mu(\Gamma)\\
            &=\lambda\mu(\Gamma)-k\mu(\Gamma)\\
            &=(\lambda-k)\mu(\Gamma)
        \end{split}
        \end{equation*}
    Again, $\big(\lambda I_n-A(\Gamma)\big)\mu(\Gamma)=(\lambda-k)\mu(\Gamma)$\\
        $\implies\big(\lambda I_n-A(\Gamma)\big)^{-1}\mu(\Gamma)=\frac{1}{\lambda-k}\mu(\Gamma)$
        Now,\begin{equation*}
        \begin{split}
            \chi_{A(\Gamma)}(\lambda)&=\mu(\Gamma)^T\big(\lambda I_n-A(\Gamma)\big)^{-1}\mu(\Gamma)\\
            &=\frac{\mu(\Gamma)^T\mu(\Gamma)}{\lambda-k}\\
                &=\frac{n}{\lambda-k}
        \end{split}
        \end{equation*}
        Therefore, $\chi_{A(\Gamma)}(\lambda)=\frac{n}{\lambda-k}.$
    \end{proof}

    \begin{lemma}\label{Lemma2}
        Let $\Gamma=(K_{1,n},\sigma,\mu)$ be a signed star with $V(\Gamma)=\{v_1,v_2,...,v_{n+1}\}$ such that $d(v_1)=n$, then \[\chi_{A(\Gamma)}(\lambda)=\frac{(n+1)\lambda+2n\mu(v_1)}{\lambda^2-n}\]
    \end{lemma}
    \begin{proof}
        Here, $A(\Gamma)=
        \begin{bmatrix}
          0 &\mu(v_2) & \cdots &\mu(v_{n+1})\\
          \mu(v_2) &0 &\cdots &0\\
          \vdots &\vdots & \ddots &\vdots\\
          \mu(v_{n+1}) &0 &\cdots &0
        \end{bmatrix}.$\\
        Let $X=diag\bigg(\frac{n+\lambda\mu(v_1)}{\mu(v_1)},\lambda+\mu(v_1),\cdots,\lambda+\mu(v_1)\bigg)$ be $(n+1)\times(n+1)$ diagonal matrix.\\ 
        Then, $\big(\lambda I_{n+1}-A(\Gamma)\big)X\mu(\Gamma)=(\lambda^2-n)\mu(\Gamma)$\\
        Thus, \begin{equation*}
            \begin{split}
                \chi_{A(\Gamma)}(\lambda)&=\mu(\Gamma)^T\big(\lambda I_n-A(\Gamma)\big)^{-1}\mu(\Gamma)\\
                &=\frac{\mu(\Gamma)^TX\mu(\Gamma)}{\lambda^2-n}\\
                &=\frac{(n+1)\lambda+2n\mu(v_1)}{\lambda^2-n}
            \end{split}
        \end{equation*}
        Therefore, $\chi_{A(\Gamma)}(\lambda)=\frac{(n+1)\lambda+2n\mu(v_1)}{\lambda^2-n}$
    \end{proof}
    
    \begin{theorem}\label{Thm1}
        Let $\Gamma_1$,$\Gamma_2$ be arbitrary signed graphs on $n_1,n_2\geq1$ nodes. Then \[f(A(\Gamma_1*\Gamma_2);x)=f(A(\Gamma_2);x))^{n_1}\cdot \prod_{i=1}^{n_1}\bigg(x-\lambda(\Gamma_1)-\chi_{A(\Gamma_2)}(x)\lambda_i(\Gamma_1)^2\bigg).\]
    \end{theorem}

    \begin{proof}
        With respect to the partition mentioned in \ref{eqn2}, the adjacency matrix of $\Gamma_1*\Gamma_2$ can be written as, 
        \[A(\Gamma_1*\Gamma_2)=\begin{bmatrix}
            A(\Gamma_1)& \mu(\Gamma_2)^T\otimes A(\Gamma_1)\phi(\Gamma_1)\\ \mu(\Gamma_2)\otimes \phi(\Gamma_1)A(\Gamma_1)& A(\Gamma_2)\otimes I_{n_1}
        \end{bmatrix}\]

        Then using Lemma \ref{lemma 0} the adjacency characteristics polynomial of $\Gamma_1*\Gamma_2$ is given by,
        \begin{equation*}
        \begin{split}
            f(A(\Gamma_1*\Gamma_2);x)&= \text{det} 
           \begin{bmatrix}
          xI_{n_1}-A(\Gamma_1)& -\mu(\Gamma_2)^T\otimes A(\Gamma_1)\phi(\Gamma_1)\\ -\mu(\Gamma_2)\otimes \phi(\Gamma_1)A(\Gamma_1)& xI_{n_1n_2}-A(\Gamma_2)\otimes I_{n_1}       
        \end{bmatrix}\\
        &= \begin{bmatrix}
                xI_{n_1}-A(\Gamma_1)& -\mu(\Gamma_2)^T\otimes A(\Gamma_1)\phi(\Gamma_1)\\ -\mu(\Gamma_2)\otimes \phi(\Gamma_1)A(\Gamma_1)& (xI_{n_2}-A(\Gamma_2))\otimes I_{n_1}
        \end{bmatrix}\\
        &= \text{det}\big(\big(xI_{n_2}-A(\Gamma_2)\big)\otimes I_{n_1}\big) \cdot \text{det}(S).
        \end{split}
        \end{equation*}
        here, 
        \begin{equation*}
            \begin{split}
                S&=\big\{xI_{n_1}-A(\Gamma_1)\big\}-\big\{\mu(\Gamma_2)^T\otimes A(\Gamma_1)\phi(\Gamma_1)\big\}\big\{\big(xI_{n_2}-A(\Gamma_2)\big)\otimes I_{n_1}\big\}^{-1}\big\{\mu(\Gamma_2)\otimes \phi(\Gamma_1)A(\Gamma_1)\big\}\\
                &=\big\{xI_{n_1}-A(\Gamma_1)\big\}-\big\{\mu(\Gamma_2)^T\otimes A(\Gamma_1)\phi(\Gamma_1)\big\}\big\{\big(xI_{n_2}-A(\Gamma_2)\big)^{-1}\otimes I_{n_1}^{-1}\big\}\big\{\mu(\Gamma_2)\otimes \phi(\Gamma_1)A(\Gamma_1)\big\}\\
                &=\big\{xI_{n_1}-A(\Gamma_1)\big\}-\big\{\mu(\Gamma_2)^T\big(xI_{n_2}-A(\Gamma_2)\big)^{-1}\mu(\Gamma_2)\big\}\big\{A(\Gamma_1)\phi(\Gamma_1)I_{n_1}\phi(\Gamma_1)A(\Gamma_1)\big\}\\
                &=\big\{xI_{n_1}-A(\Gamma_1)\big\}-\chi_{\Gamma_2}(x)\big\{A(\Gamma_1)\big\}^2\\
            \end{split}
        \end{equation*}
        Therefore,
        \begin{equation*}
            \begin{split}
                \text{det}(S)&=\text{det}\big\{\big\{xI_{n_1}-A(\Gamma_1)\big\}-\chi_{\Gamma_2}(x)\big\{A(\Gamma_1)\big\}^2\big\}\\
                &=\prod_{i=1}^{n_1}\big(x-\lambda_i(\Gamma_1)-\chi_{\Gamma_2}(x)\lambda_i(\Gamma_1)^2\big)
            \end{split}
        \end{equation*}
        Here we use the fact that if $\lambda$ is an eigenvalue of a matrix $A$ and $f(A)$ is a polynomial of $A$ then $f(\lambda)$ is an eigenvalue of $f(A).$ Also,\\
       \begin{equation*}
           \text{det}\big(\big(xI_{n_2}-A(\Gamma_2)\big)\otimes I_{n_1}\big)
           =\text{det}\big(\big(xI_{n_2}-A(\Gamma_2)\big)^{n_1}\cdot \text{det}\big(I_{n_1}\big)^{n_2}
           =\text{det}\big((xI_{n_2}-A(\Gamma_2)\big)^{n_1}
       \end{equation*} 
       \[\therefore,~ f(A(\Gamma_1*\Gamma_2);x)=f\big(A(\Gamma_2);x)\big)^{n_1}\cdot \prod_{i=1}^{n_1}\bigg(x-\lambda(\Gamma_1)-\chi_{A(\Gamma_2)}(x)\lambda_i(\Gamma_1)^2\bigg).~~~~~~~~~~~~~~~~~\]
    \end{proof}

    \begin{prop}\label{Prop1}
        Let $\Gamma_1=(G_1,\sigma_1,\mu_1)$ be any signed graph on $n_1$ nodes and $\Gamma_2=(G_2,\sigma_2,\mu_2)$ be $(r,k)$ co-regular signed graph on $n_2$ nodes, where $n_1\geq1,~n_2\geq2$ and $r\geq1$. Suppose that the spectrum of$~~\Gamma_1$ and $\Gamma_2$ are $\{\alpha_1,\alpha_2,\cdots,\alpha_{n_1}\}$, and $\{\beta_1,\beta_2\cdots,\beta_{n_2}\}$ respectively and multiplicity of eigenvalue $k$ of $\Gamma_2$ is $p$ then the spectrum of $\Gamma_1*\Gamma_2$ is given by\\
        \begin{enumerate}[label=(\roman*)]
            \item $2$ eigenvalues $\frac{\alpha_i+k\pm\sqrt{(\alpha_i-k)^2+4n_2\alpha_i^2}}{2}$, for each $i=1,2,\cdots,n_1.$
            \item $n_1(n_2-p)$ eigenvalues $\beta_i$ each with multiplicity $n_1$ for every eigenvalue $\beta_i(\neq k)$ of $\Gamma_2.$
            \item eigenvalue $k$ with multiplicity $n_1(p-1).$
        \end{enumerate}
    \end{prop}
    \begin{proof}
        By Lemma \ref{Lemma1}, \[\chi_{A(\Gamma)}(\lambda)=\frac{n}{\lambda-k}\]
        The only pole of $\chi_{A(\Gamma)}(\lambda)$ is $\lambda=k$, which is equivalent to the eigenvalue $\lambda=k$ of $\Gamma_2$. By Theorem \ref{Thm1} the spectrum of $\Gamma_1*\Gamma_2$ is given by,
        \begin{enumerate}[label=(\roman*)]
            \item Then $2n_1$ eigenvalues are obtained by solving\\
               \begin{equation*}
               \begin{split}
                   &x-\alpha_i-\frac{n_2}{x-k}\alpha_i^2=0\\
                    \implies& x^2-(k+\alpha_i)x+\alpha_ik-n_2\alpha_i^2=0\\
                    \implies&x=\frac{\alpha_i+k\pm\sqrt{(\alpha_i-k)^2+4n_2\alpha_i^2}}{2},
               \end{split}
                \end{equation*}
            for each value of $\alpha_i(i=1,2,\cdots,n_1)$ of $\Gamma_1.$
            \item the other $n_1(n_2-p)$ eigenvalues are $\beta_i$ each with multiplicity $n_1$ for every eigenvalue $\beta_i(\neq k)$ of $\Gamma_2.$
            \item The remaining $n_1(p-1)$ eigenvalues must be equal to the pole $\lambda=k$ of $\chi_{A(\Gamma_2)}(\lambda).$
        \end{enumerate}
    \end{proof}

    \begin{prop}
        Let $\Gamma_1=(G_1,\sigma_1,\mu_1)$ be any signed graph on $n_1$ nodes and $\Gamma_2=(K_{1,k},\sigma_2,\mu_2)$ be a signed star on $n_2+1$ nodes with $V(\Gamma_2)=\{v_1,v_2,\cdots,v_{n_2+1}\}$ where $d(v_1)=n$. Suppose that the spectrum of $\Gamma_1$ is $\{\alpha_1,\alpha_2,\cdots,\alpha_{n_1}\}$ then the spectrum of $\Gamma_1*\Gamma_2$ is given by
        \begin{enumerate}[label=(\roman*)]
            \item The eigenvalue $0$ with multiplicity $n_1(n_2-1).$
            \item $3n_1$ eigenvalues obtained by solving the equation $x^3-\alpha_ix^2-\{(n+1)\alpha_i^2+n\}x+n\alpha_i(1-2\alpha_i\mu(v_1))=0$ for each $\alpha_i(i=1,2,\cdots,n_1)$ of $\Gamma_1.$
        \end{enumerate}
    \end{prop}
    \begin{proof}
        The spectrum of $\Gamma_2$ is $\big(-\sqrt{n_2},\sqrt{n_2},0^{(n_2-1)}\big)$. By Lemma \ref{Lemma2},\[\chi_{A(\Gamma)}(\lambda)=\frac{(n+1)\lambda+2n\mu(v_1)}{\lambda^2-n}\]
        The poles of $\chi_{A(\Gamma_2)}(\lambda)$ are $\lambda=\pm\sqrt{n}$ which is equivalent to the maximal and minimal eigenvalues of $\Gamma_2.$ By Theorem \ref{Thm1}, the spectrum of $\Gamma_1*\Gamma_2$ is given by,
        \begin{enumerate}[label=(\roman*)]
            \item The eigenvalue $0$ with multiplicity $n_1(n_2-1)$.
            \item The other $3n_1$ eigenvalues are obtained by solving the cubic equation,
            \begin{equation*}
                \begin{split}
                    &x-\alpha_i-\chi_{A(\Gamma_2)}(x)\alpha_i^2=0\\
                    \implies &x-\alpha_i-\frac{(n+1)x+2n\mu(v_1)}{x^2-n}\alpha_i^2=0\\
                    \implies& x^3-\alpha_ix^2-\{(n+1)\alpha_i^2+n\}x+n\alpha_i(1-2\alpha_i\mu(v_1))=0
                \end{split}
            \end{equation*}
            for each eigenvalue $\alpha_i(i=1,2,\cdots,n_1)$ of $\Gamma_1.$
        \end{enumerate}
    \end{proof}

    \noindent We can use the neighbourhood corona product to construct may $A-$co-spectral signed graphs, as stated in the following corollary of Theorem \ref{Thm1}.

    \begin{corollary}
       Let $\Gamma_1$ and $\Gamma_2$ be $A$-co-spectral  signed graphs, and $\Gamma$ an arbitrary signed graph. Then
       \begin{enumerate}[label=(\roman*)]
           \item $\Gamma_1*\Gamma$ and $\Gamma_2*\Gamma$ are $A$-co-spectral.
           \item $\Gamma*\Gamma_1$ and $\Gamma*\Gamma_2$ are $A$-co-spectral if $\chi_{A(\Gamma_1)}(x)=\chi_{A(\Gamma_2)}(x)$
       \end{enumerate}
    \end{corollary}

\subsubsection{Q-spectra of signed graph}

    \begin{lemma} \label{Lemma 3}
        Let $\Gamma=(G,\sigma,\mu)$ be a co-regular signed graph of order $n$ with co-regularity pair $(r,k)$ then the signed $Q(\Gamma)$-coronal of $\Gamma$ is \[\chi_{Q(\Gamma)}(\lambda)=\frac{n}{\lambda-r-k}\] 
    \end{lemma}
    \begin{proof}
        Proceeding as on Lemma \ref{Lemma1} we can say that $Q(\Gamma)\mu(\Gamma)=(r+k)\mu(\Gamma.)$ Thus,
        \begin{equation*}
            \begin{split}
                \chi_{Q(\Gamma)}(\lambda)&=\mu(\Gamma)^T(\lambda I_n-Q(\Gamma))^{-1}\mu(\Gamma)\\
                &=\frac{\mu(\Gamma)^T\mu(\Gamma)}{\lambda-r-k}\\
                &=\frac{n}{\lambda-r-k}
            \end{split}
        \end{equation*}
    \end{proof}

    \begin{lemma}\label{lemma 4}
        Let $\Gamma=(K_{1,n},\sigma,\mu)$ be a signed star with $V(\Gamma)=\{v_1,v_2,\cdots,v_{n+1}\}$ such that $d(v_1)=n$, then  \[
 \chi_{Q(\Gamma)}(\lambda)=\frac{(n+1)\lambda-(n^2+1)+2n\mu(v_1)}{\lambda(\lambda-(n+1))}
 \]
        
    \end{lemma}
     \begin{proof}
         We have $Q(\Gamma)=
        $\begin{math}
        \begin{bmatrix}
           n &\mu(v_2) &\mu(v_3) &\cdots &\mu(v_{n+1}) \\
           \mu(v_2)   &1 &0 &\cdots &0 \\
           \vdots &\vdots &\vdots &   &\vdots\\
           \mu(v_{n+1}) &0  &0 &\cdots &1
         \end{bmatrix}
        \end{math}\\
        Taking $X=diag\Big(\frac{(\lambda-1) \mu(v_1)+n}{\mu(v_1)}, \lambda-n+\mu(v_1), \cdots, \lambda-n+\mu(v_1)\Big)$ be $(n+1)\times(n+1)$ diagonal matrix with first diagonal entry as $\frac{(\lambda-1)\mu(v_1)+n}{\mu(n_1)}$ and  remaining $n$ diagonal entries as $\lambda-n+\mu(v_1).$
        \\Then $(\lambda I_{n+1}-L(\Gamma))X\mu(\Gamma)=\lambda(\lambda-(n+1))\mu(\Gamma).$ Thus, 
        \begin{align*}
             \chi_{Q(\Gamma)}(\lambda)=&\mu(\Gamma)^T\big(\lambda I_n - Q(\Gamma)\big)^{-1}\mu(\Gamma)\\ \\
             =&\frac{\mu(\Gamma)^TX\mu(\Gamma)}{\lambda\big(\lambda-(n+1)\big)}\\ \\
             =&\frac{(n+1)\lambda-(n^2+1)+2n\mu(v_1)}{\lambda\big(\lambda-(n+1)\big)}
        \end{align*}\\
    \end{proof}

\begin{theorem} \label{Thm2}
    Let $\Gamma_1$ be a $r_1$-regular signed graph on $n_1$ nodes and $\Gamma_2$ be any signed graph on $n_2$ nodes, where $n_1\geq2,n_2\geq1$ and $r_1\geq1.$ Then,
    \[f\big(Q(\Gamma_1*\Gamma_2);x\big)=\big(f\big(Q(\Gamma_2);x-r_1\big)\big)^{n_1}\cdot \prod_{i=1}^{n_1}\bigg(x-n_2r_1-\gamma_i(\Gamma_1)-\chi_{Q(\Gamma_2)}(x-r_1)(\gamma_i(\Gamma_1)-r_1)^2\bigg).\]
\end{theorem}
\begin{proof}
    $0_{n_1\times n_1}$ is the all $0$ $n_1\times n_1$ matrix. With respect to the partition mentioned in \ref{eqn2}, the nodes degree diagonal matrix and signless Laplacian matrix of $\Gamma_1*\Gamma_2$ can be written as,
    \[D(\Gamma_1*\Gamma_2)=\begin{bmatrix}
        (n_2+1)D(\Gamma_1)& \mu(\Gamma_2)^T\otimes 0_{n_1\times n_1}\\
        \mu(\Gamma_2)\otimes 0_{n_1\times n_1}& D(\Gamma_2)\otimes I_{n_1}+I_{n_2}\otimes D(\Gamma_1)
    \end{bmatrix}\]
    and
    \begin{equation*}
    \begin{split}
         Q(\Gamma_1*\Gamma_2)&=D(\Gamma_1*\Gamma_2)+A(\Gamma_1*\Gamma_2)\\
        & =\begin{bmatrix}
                n_2D(\Gamma_1)+Q(\Gamma_1)& \mu(\Gamma_2)^T\otimes A(\Gamma_1)\phi(\Gamma_1)\\
                \mu(\Gamma_2)\otimes \phi(\Gamma_1)A(\Gamma_1)& Q(\Gamma_2)\otimes I_{n_1}+I_{n_2}\otimes D(\Gamma_1)
          \end{bmatrix}
    \end{split}
    \end{equation*}
    Since $\Gamma_1$ is $r_1$- regular, so $D(\Gamma_1)=r_1I_{n_1}.$\\
    Now, using Lemma \ref{lemma 0}
    \begin{equation*}
        \begin{split}
            f(Q(\Gamma_1*\Gamma_2);x)&=\text{det}\begin{bmatrix}
               (x-n_2r_1)I_{n_1}-Q(\Gamma_1)& -\mu(\Gamma_2)^T\otimes A(\Gamma_1)\phi(\Gamma_1)\\ -\mu(\Gamma_2)\otimes \phi(\Gamma_1)A(\Gamma_1)& (xI_{n_1n_2}-Q(\Gamma_2)\otimes I_{n_1}-I_{n_2}\otimes r_1I_{n_1})
               \end{bmatrix}\\
            &=\text{det}\begin{bmatrix}
               (x-n_2r_1)I_{n_1}-Q(\Gamma_1)& -\mu(\Gamma_2)^T\otimes A(\Gamma_1)\phi(\Gamma_1)\\ -\mu(\Gamma_2)\otimes \phi(\Gamma_1)A(\Gamma_1)& \big((x-r_1)I_{n_2}-Q(\Gamma_2)\big)\otimes I_{n_1}
            \end{bmatrix}\\
            &=\text{det}\bigg\{\big((x-r_1)I_{n_2}-Q(\Gamma_2)\big)\otimes I_{n_1}\bigg\} \cdot \text{det}(S)\\
        \end{split}
    \end{equation*}
    here \begin{equation*}
        \begin{split}
            S&=\big((x-n_2r_1)I_{n_1}-Q(\Gamma_1)\big)-\bigg[\big(\mu(\Gamma_2)^T\otimes A(\Gamma_1)\phi(\Gamma_1)\big)\bigg(\big((x-r_1)I_{n_2}-Q(\Gamma_2)\big)\otimes I_{n_1}\bigg)^{-1}\\ &~~~~\big(\mu(\Gamma_2)\otimes \phi(\Gamma_1)A(\Gamma_1)\big)\bigg]\\
            &=\big((x-n_2r_1)I_{n_1}-Q(\Gamma_1)\big)-\bigg[\big(\mu(\Gamma_2)^T\otimes A(\Gamma_1)\phi(\Gamma_1)\big)\bigg(\big((x-r_1)I_{n_2}-Q(\Gamma_2)\big)^{-1}\otimes I_{n_1}^{-1}\bigg)\\ &~~~~\big(\mu(\Gamma_2)\otimes \phi(\Gamma_1)A(\Gamma_1)\big)\bigg]\\
            &=\big((x-n_2r_1)I_{n_1}-Q(\Gamma_1)\big)-\big\{\mu(\Gamma_2)^T\big((x-r_1)I_{n_1}-Q(\Gamma_2)\big)^{-1}\mu(\Gamma_2)\big\}\cdot\big\{A(\Gamma_1)\phi(\Gamma_1)I_{n_1}\phi(\Gamma_1)(\Gamma_1)\big\}\\
            &=\big((x-n_2r_1)I_{n_1}-Q(\Gamma_1)\big)-\chi_{\Gamma+2}(x-r_1)A(\Gamma_1)^2
        \end{split}
    \end{equation*}
    Therefore,\begin{equation*}
        \begin{split}
            \text{det}(S)&=\text{det}\big\{\big((x-n_2r_1)I_{n_1}-Q(\Gamma_1)\big)-\chi_{Q\Gamma_2)}(x-r_1)A(\Gamma_1)^2\big\}\\
            &=\prod_{i=1}^{n_1}\big\{\big((x-n_2r_1)-\gamma_i(\Gamma_1)\big)-\chi_{Q(\Gamma_2)}(x-r_1)\big(\gamma_i(\Gamma_1)-r_1\big)^2\big\}
        \end{split}
    \end{equation*}
    Here we use the fact that $\lambda$ is an eigenvalue of a matrix $A(G)$ with eigenvector $v$, iff $\lambda+r_1$ is an eigenvalue of $Q(G)$ with the same eigenvector $v$.
    Also, \[\text{det}\big\{\big((x-r_1)I_{n_2}-Q(\Gamma_2)\big) \otimes I_{n_1}\big\}=\text{det}\big((x-r_1)I_{n_2}-Q(\Gamma_2)\big)^{n_1}\otimes ~\text{det}(I_{n_1})^{n_2}=\big(f(Q(\Gamma_2);x-r_1)\big)^{n_1}\]
    Hence,
    \[f\big(Q(\Gamma_1*\Gamma_2);x\big)=\big(f\big(Q(\Gamma_2);x-r_1\big)\big)^{n_1}\cdot \prod_{i=1}^{n_1}\bigg(x-n_2r_1-\gamma_i(\Gamma_1)-\chi_{Q(\Gamma_2)}(x-r_1)(\gamma_i(\Gamma_1)-r_1)^2\bigg).\]
\end{proof}

\begin{prop}\label{Prop 3}
    Let $\Gamma_1=(G_1,\sigma_1,\mu_1)$ be a $r_1$-regular signed graph on $n_1$ nodes and $\Gamma_2=(G_2,\sigma_2,\mu_2)$ be $(r_2,k_2)$ co-regular graph on $n_2$ nodes. Suppose that the spectrum of $Q(\Gamma_1)$ and $Q(\Gamma_2)$ are $(\alpha_1,\alpha_2,\cdots,\alpha_{n_1})$ and $(\beta_1,\beta_2,\cdots,\beta_{n_2})$ respectively and multiplicity of eigenvalue $(r_2+k_2)$ of $Q(\Gamma_2)$ is $p$ then the spectrum of $Q(\Gamma_1*\Gamma_2)$ is given by
    \begin{enumerate}[label=(\roman*)]
        \item $2n_1$ eigenvalues $\frac{(r_1+r_2+k_2+r_1n_2+\alpha_i)\pm\sqrt{(r_1+r_2+k_2-r_1n_2-\alpha_i)^2+4n_2(\alpha_i-r_1)^2}}{2}$ for each eigenvalue $\alpha_i(i=1,2,\cdots,n_1)$ of $\Gamma_1.$
        \item $n_1(n_2-p)$ eigenvalues $\beta_i+r_1$ each with multiplicity $n_1$ for very eigenvalue $\beta_i(\neq r_2+k_2)$ of $Q(\Gamma_2)$.
        \item Eigenvalue $2r_2+k_2$ with multiplicity $n_1(p-1).$
    \end{enumerate}
\end{prop}
\begin{proof}
    By Lemma \ref{Lemma 3}, \[\chi_{Q(\Gamma_2)}(\lambda-r_1)=\frac{n_2}{\lambda-r_1-r_2-k_2}\]
    The only pole of $\chi_{Q(\Gamma_2)}(\lambda-r_1)$ is $\lambda=r_1-r_2-k_2$ which is equivalent to the eigenvalue $\lambda-r_1=r_2+k_2$ of $Q(\Gamma_2).$ By Theorem \ref{Thm2} the spectrum of $Q(\Gamma_1*\Gamma_2)$ is given by
    \begin{enumerate}
        \item The $2n_1$, eigenvalues obtained by solving
        \begin{equation*}
            \begin{split}
                &x-n_2r_1-\alpha_i-\frac{n_2}{x-r_1-r_2-k_2}(\alpha_i+r_1)^2=0\\
                \implies & x^2-r_1x-r_2x-k_2x-n_2r_1x+n_2r_1^2+r_1r_2n_2+n_2r_1k_2-\alpha_ix+r_1\alpha_i+r_2\alpha_i+k_2\alpha_i\\ &-n_2(\alpha_i-r_1)^2=0\\
                \implies & x^2-(r_1+r_2+k_2+n_2r_1+\alpha_i)x+r_1^2n_2+r_1r_2n_2+r_1n_2k_2+r_1\alpha_i+r_2\alpha_i+k_2\alpha_i-\\ &n_2(\alpha_i-r_1)^2=0\\
                \implies & x=\frac{(r_1+r_2+k_2+n_2r_1+\alpha_i)\pm\sqrt{(r_1+r_2+k_2-r_1n_2-\alpha_i)^2+4n_2(\alpha_i-r_1)^2}}{2}
            \end{split}
        \end{equation*}
        for each eigenvalue $\alpha_i(i=1,2,\cdots,n_1)$ of $Q(\Gamma_1)$.
        \item The other $n_1(n_2-p)$ eigenvalues are $\beta_i+r_1$ each with multiplicity $n_1$ for every eigenvalue $\beta_i(\neq r_2+k_2)$ of $\Gamma_2.$
        \item The remaining $n_1(p-1)$ eigenvalues must equal the only pole $\lambda=2r_2+k_2$ of $\chi_{Q(\Gamma_2)}(\lambda-r_1).$
    \end{enumerate}
\end{proof}

\begin{prop}\label{Prop 4}
    Let $\Gamma_1=(G_1,\sigma_1,\mu_1)$ be any signed graph on $n_1$ nodes and $\Gamma_2=(K_{1,n},\sigma_2,\mu_2)$ be a signed star on $n+1$ nodes with $V(\Gamma_2)=\{v_1,v_2,\cdots,v_{n+1}\}$ where $d(v_1)=n$. Suppose that the spectrum of $Q(\Gamma_1)$ is $\{\alpha_1,\alpha_2,\cdots,\alpha_{n_1}\}$ then the spectrum of $Q(\Gamma_1*\Gamma_2)$ is given by
    \begin{enumerate}[label=(\roman*)]
        \item The eigenvalue $1+r_1$ with multiplicity $n_1(n-1).$
        \item $3n_1$ eigenvalues obtained from the roots of the equation $x^3-\{r_1(n+2)+(n+1)+\alpha_i\}x^2+\{r_1(2n+1)(r_1+1)+n^2r_1+(2r_1+n+1)\alpha_i-(n+1)(\alpha_i-r_1)^2\}x-nr_1^3-(n+1)nr_1^2-r_1^2\alpha_i-(n+1)r_1\alpha_i+\{r_1(n+1)+(n^2+1)-2n\mu_2(v_1)\}(\alpha_i-r_1)^2=0$ for each eigenvalue $\alpha_i(i=1,2,\cdots,n_1)$ of $Q(\Gamma_1).$
    \end{enumerate}
    
\end{prop}
\begin{proof}
    The spectrum of $Q(\Gamma_2)$ is $(0,1^{n-1},n+1)$. By Lemma \ref{lemma 4}, \[\chi_{Q(\Gamma_2)}(x-r_1)=\frac{(n+1)(x-r_1)-(n^2+1)+2n\mu(v_1)}{(x-r_1)\big(x-r_1-(n+1)\big)}\]
    Two poles of $\chi_{Q(\Gamma_2)}(x-r_1)$ are $\lambda=r_1$ and $\lambda=r_1+n+2$ which is equivalent to maximum eigenvalue $\lambda-r_1=n+1$ and minimum eigenvalue $\lambda-r_1=0$ of $Q(\Gamma_2)$. By Theorem \ref{Thm2}, the spectrum of $Q(\Gamma_1*\Gamma_2)$ is given by:
    \begin{enumerate}[label=(\roman*)]
        \item The eigenvalue $1+r_1$ with multiplicity $n_1(n-1).$
        \item The other $3n_1$ eigenvalues are obtained by solving
        \begin{equation*}
            \begin{split}
            &x-nr_1-\alpha_i-\frac{(n+1)(x-r_1)-(n^2+1)+2n\mu_2(v_1)}{(x-r_1)\big(x-r_1-(n+1)\big)}(\alpha_i-r_1)^2=0\\
            \implies& x^3-\{r_1(n+2)+(n+1)+\alpha_i\}x^2+\{r_1(2n+1)(r_1+1)+n^2r_1+(2r_1+n+1)\alpha_i-\\ &(n+1)(\alpha_i-r_1)^2\}x-nr_1^3-(n+1)nr_1^2-r_1^2\alpha_i-(n+1)r_1\alpha_i+\{r_1(n+1)+(n^2+1)\\ &-2n\mu_2(v_1)\}(\alpha_i-r_1)^2=0.
            \end{split}
        \end{equation*}
        for each eigenvalue $\alpha_i(i=1,2,\cdots,n_1)$ of $Q(\Gamma_1).$
    \end{enumerate}
\end{proof}

\noindent Using Theorem \ref{Thm2}  we can construct many pairs of $Q$-co-spectral signed graphs, as stated in the following corollary.

    \begin{corollary}
       Let $\Gamma_1$ and $\Gamma_2$ be $Q$-co-spectral  signed graphs, and $\Gamma$ an arbitrary signed graph. Then
       \begin{enumerate}[label=(\roman*)]
           \item $\Gamma_1*\Gamma$ and $\Gamma_2*\Gamma$ are $Q$-co-spectral.
           \item $\Gamma*\Gamma_1$ and $\Gamma*\Gamma_2$ are $Q$-co-spectral if $\chi_{Q(\Gamma_1)}(x)=\chi_{Q(\Gamma_2)}(x)$
       \end{enumerate}
    \end{corollary}

\subsubsection{L-spectra of signed graph}

    \begin{lemma} \label{Lemma 5}
        Let $\Gamma=(G,\sigma,\mu)$ be a co-regular signed graph of order $n$ with co-regularity pair $(r,k)$ then the signed $L(\Gamma)$-coronal of $\Gamma$ is \[\chi_{L(\Gamma)}(\lambda)=\frac{n}{\lambda-r+k}\] 
    \end{lemma}
    \begin{proof}
        The proof is similar to that of Lemma \ref{Lemma 3}
    \end{proof}

    \begin{lemma}\label{lemma 6}
        Let $\Gamma=(K_{1,n},\sigma,\mu)$ be a signed star with $V(\Gamma)=\{v_1,v_2,\cdots,v_{n+1}\}$ such that $d(v_1)=n$, then  \[
 \chi_{L(\Gamma)}(\lambda)=\frac{(n+1)\lambda-(n^2+1)+2n\mu(v_1)}{\lambda(\lambda-(n+1))}
 \]
        
    \end{lemma}
     \begin{proof}
         We have $L(\Gamma)=
        $\begin{math}
        \begin{bmatrix}
           n &-\mu(v_2) &-\mu(v_3) &\cdots &-\mu(v_{n+1}) \\
           -\mu(v_2)   &1 &0 &\cdots &0 \\
           \vdots &\vdots &\vdots &   &\vdots\\
           -\mu(v_{n+1}) &0  &0 &\cdots &1
         \end{bmatrix}
        \end{math}\\
        Taking $X=diag\Big(\frac{(\lambda-1) \mu(v_1)-n}{\mu(v_1)}, \lambda-n-\mu(v_1), \cdots, \lambda-n-\mu(v_1)\Big)$ be $(n+1)\times(n+1)$ diagonal matrix with first diagonal entry as $\frac{(\lambda-1)\mu(v_1)-n}{\mu(n_1)}$ and  remaining $n$ diagonal entries as $\lambda-n-\mu(v_1).$ And proceeding as on Lemma \ref{lemma 4} we will get the required result.
    \end{proof}
    
\begin{theorem} \label{Thm 3}
     Let $\Gamma_1$ be an $r_1$-regular signed graph on $n_1$ nodes and $\Gamma_2$ be any signed graph on $n_2$ nodes, where $n_1\geq2,n_2\geq1$ and $r_1\geq1.$ Then,
    \[f\big(L(\Gamma_1*\Gamma_2);x\big)=\big(f\big(L(\Gamma_2);x-r_1\big)\big)^{n_1}\cdot \prod_{i=1}^{n_1}\bigg(x-n_2r_1-\nu_i(\Gamma_1)-\chi_{L(\Gamma_2)}(x-r_1)(r_1-\nu_i(\Gamma_1))^2\bigg).\]
\end{theorem}

\begin{proof}
    $0_{n_1\times n_1}$ is the all $0$ $n_1\times n_1$ matrix. With respect to the partition mentioned in \ref{eqn2}, the nodes degree diagonal matrix and signed Laplacian matrix of $\Gamma_1*\Gamma_2$ can be written as,
    \[D(\Gamma_1*\Gamma_2)=\begin{bmatrix}
        (n_2+1)D(\Gamma_1)& \mu(\Gamma_2)^T\otimes 0_{n_1\times n_1}\\
        \mu(\Gamma_2)\otimes 0_{n_1\times n_1}& D(\Gamma_2)\otimes I_{n_1}+I_{n_2}\otimes D(\Gamma_1)
    \end{bmatrix}\]
    and
    \begin{equation*}
    \begin{split}
         L(\Gamma_1*\Gamma_2)&=D(\Gamma_1*\Gamma_2)-A(\Gamma_1*\Gamma_2)\\
        & =\begin{bmatrix}
                n_2D(\Gamma_1)+L(\Gamma_1)& -\mu(\Gamma_2)^T\otimes A(\Gamma_1)\phi(\Gamma_1)\\
                -\mu(\Gamma_2)\otimes \phi(\Gamma_1)A(\Gamma_1)& L(\Gamma_2)\otimes I_{n_1}+I_{n_2}\otimes D(\Gamma_1)
          \end{bmatrix}
    \end{split}
    \end{equation*}
    Since $\Gamma_1$ is $r_1$- regular, so $D(\Gamma_1)=r_1I_{n_1}.$\\
    Now, using Lemma \ref{lemma 0}
    \begin{equation*}
        \begin{split}
            f(L(\Gamma_1*\Gamma_2);x)&=\text{det}\begin{bmatrix}
               (x-n_2r_1)I_{n_1}+L(\Gamma_1)& \mu(\Gamma_2)^T\otimes A(\Gamma_1)\phi(\Gamma_1)\\ \mu(\Gamma_2)\otimes \phi(\Gamma_1)A(\Gamma_1)& (xI_{n_1n_2}-L(\Gamma_2)\otimes I_{n_1}-I_{n_2}\otimes r_1I_{n_1})
               \end{bmatrix}\\
            &=\text{det}\begin{bmatrix}
               (x-n_2r_1)I_{n_1}+L(\Gamma_1)& -\mu(\Gamma_2)^T\otimes A(\Gamma_1)\phi(\Gamma_1)\\ -\mu(\Gamma_2)\otimes \phi(\Gamma_1)A(\Gamma_1)& \big((x-r_1)I_{n_2}-L(\Gamma_2)\big)\otimes I_{n_1}
            \end{bmatrix}\\
            &=\text{det}\bigg\{\big((x-r_1)I_{n_2}-L(\Gamma_2)\big)\otimes I_{n_1}\bigg\} \cdot \text{det}(S)\\
        \end{split}
    \end{equation*}
    here \begin{equation*}
        \begin{split}
            S&=\big((x-n_2r_1)I_{n_1}+L(\Gamma_1)\big)-\bigg[\big(\mu(\Gamma_2)^T\otimes A(\Gamma_1)\phi(\Gamma_1)\big)\bigg(\big((x-r_1)I_{n_2}-L(\Gamma_2)\big)\otimes I_{n_1}\bigg)^{-1}\\ &~~~~\big(\mu(\Gamma_2)\otimes \phi(\Gamma_1)A(\Gamma_1)\big)\bigg]\\
            &=\big((x-n_2r_1)I_{n_1}+L(\Gamma_1)\big)-\bigg[\big(\mu(\Gamma_2)^T\otimes A(\Gamma_1)\phi(\Gamma_1)\big)\bigg(\big((x-r_1)I_{n_2}-L(\Gamma_2)\big)^{-1}\otimes I_{n_1}^{-1}\bigg)\\ &~~~~\big(\mu(\Gamma_2)\otimes \phi(\Gamma_1)A(\Gamma_1)\big)\bigg]\\
            &=\big((x-n_2r_1)I_{n_1}+L(\Gamma_1)\big)-\big\{\mu(\Gamma_2)^T\big((x-r_1)I_{n_1}-L(\Gamma_2)\big)^{-1}\mu(\Gamma_2)\big\}\cdot\big\{A(\Gamma_1)\phi(\Gamma_1)I_{n_1}\phi(\Gamma_1)(\Gamma_1)\big\}\\
            &=\big((x-n_2r_1)I_{n_1}+L(\Gamma_1)\big)-\chi_{\Gamma+2}(x-r_1)A(\Gamma_1)^2
        \end{split}
    \end{equation*}
    Therefore,\begin{equation*}
        \begin{split}
            \text{det}(S)&=\text{det}\big\{\big((x-n_2r_1)I_{n_1}+L(\Gamma_1)\big)-\chi_{L(\Gamma_2)}(x-r_1)A(\Gamma_1)^2\big\}\\
            &=\prod_{i=1}^{n_1}\big\{\big((x-n_2r_1)-\nu_i(\Gamma_1)\big)-\chi_{L(\Gamma_2)}(x-r_1)\big(r_1-\nu_i(\Gamma_1)\big)^2\big\}
        \end{split}
    \end{equation*}
    Here we use the fact that $\lambda$ is an eigenvalue of a matrix $A(G)$ with eigenvector $v$, iff $r_1-\lambda$ is an eigenvalue of $L(G)$ with the same eigenvector $v$.
    Also, \[\text{det}\big\{\big((x-r_1)I_{n_2}-L(\Gamma_2)\big) \otimes I_{n_1}\big\}=\text{det}\big((x-r_1)I_{n_2}-L(\Gamma_2)\big)^{n_1}\otimes ~\text{det}(I_{n_1})^{n_2}=\big(f(L(\Gamma_2);x-r_1)\big)^{n_1}\]
    Hence,
    \[f\big(L(\Gamma_1*\Gamma_2);x\big)=\big(f\big(Q(\Gamma_2);x-r_1\big)\big)^{n_1}\cdot \prod_{i=1}^{n_1}\bigg(x-n_2r_1-\nu_i(\Gamma_1)-\chi_{L(\Gamma_2)}(x-r_1)(r_1-\nu_i(\Gamma_1))^2\bigg).\]
\end{proof}

\begin{prop}\label{Prop 5}
    Let $\Gamma_1=(G_1,\sigma_1,\mu_1)$ be a $r_1$-regular signed graph on $n_1$ nodes and $\Gamma_2=(G_2,\sigma_2,\mu_2)$ be $(r_2,k_2)$ co-regular graph on $n_2$ nodes. Suppose that the spectrum of $L(\Gamma_1)$ and $L(\Gamma_2)$ are $(\alpha_1,\alpha_2,\cdots,\alpha_{n_1})$ and $(\beta_1,\beta_2,\cdots,\beta_{n_2})$ respectively and multiplicity of eigenvalue $(r_2+k_2)$ of $L(\Gamma_2)$ is $p$ then the spectrum of $L(\Gamma_1*\Gamma_2)$ is given by
    \begin{enumerate}[label=(\roman*)]
        \item $2n_1$ eigenvalues $\frac{(r_1+r_2-k_2+r_1n_2+\alpha_i)\pm\sqrt{(r_1+r_2-k_2-r_1n_2-\alpha_i)^2+4n_2(r_1-\alpha_i)^2}}{2}$ for each eigenvalue $\alpha_i(i=1,2,\cdots,n_1)$ of $\Gamma_1.$
        \item $n_1(n_2-p)$ eigenvalues $\beta_i+r_1$ each with multiplicity $n_1$ for very eigenvalue $\beta_i(\neq r_2-k_2)$ of $L(\Gamma_2)$.
        \item Eigenvalue $2r_2+k_2$ with multiplicity $n_1(p-1).$
    \end{enumerate}
\end{prop}
\begin{proof}
   The proof is similar to that of Proposition \ref{Prop 3}
\end{proof}

\begin{prop} \label{Prop 6}
    Let $\Gamma_1=(G_1,\sigma_1,\mu_1)$ be any signed graph on $n_1$ nodes and $\Gamma_2=(K_{1,n},\sigma_2,\mu_2)$ be a signed star on $n+1$ nodes with $V(\Gamma_2)=\{v_1,v_2,\cdots,v_{n+1}\}$ where $d(v_1)=n$. Suppose that the spectrum of $L(\Gamma_1)$ is $\{\alpha_1,\alpha_2,\cdots,\alpha_{n_1}\}$ then the spectrum of $L(\Gamma_1*\Gamma_2)$ is given by
    \begin{enumerate}[label=(\roman*)]
        \item The eigenvalue $1+r_1$ with multiplicity $n_1(n-1).$
        \item $3n_1$ eigenvalues obtained from the roots of the equation $x^3-\{r_1(n+2)+(n+1)+\alpha_i\}x^2+\{r_1(2n+1)(r_1+1)+n^2r_1+(2r_1+n+1)\alpha_i-(n+1)(r_1-\alpha_i)^2\}x-nr_1^3-(n+1)nr_1^2-r_1^2\alpha_i-(n+1)r_1\alpha_i+\{r_1(n+1)+(n^2+1)-2n\mu_2(v_1)\}(r_1-\alpha_i)^2=0$ for each eigenvalue $\alpha_i(i=1,2,\cdots,n_1)$ of $L(\Gamma_1).$
    \end{enumerate}
    
\end{prop}
\begin{proof}
   The proof is similar to that of Proposition \ref{Prop 4}
\end{proof}

\noindent Using Theorem \ref{Thm 3} we can construct many pairs of $L-$co-spectral signed graphs, as stated in the following corollary of Theorem \ref{Thm1}.

    \begin{corollary}
       Let $\Gamma_1$ and $\Gamma_2$ be $L$-co-spectral  signed graphs, and $\Gamma$ an arbitrary signed graph. Then
       \begin{enumerate}[label=(\roman*)]
           \item $\Gamma_1*\Gamma$ and $\Gamma_2*\Gamma$ are $L$-co-spectral.
           \item $\Gamma*\Gamma_1$ and $\Gamma*\Gamma_2$ are $L$-co-spectral if $\chi_{L(\Gamma_1)}(x)=\chi_{L(\Gamma_2)}(x)$
       \end{enumerate}
    \end{corollary}

\begin{example}
    Let $\Gamma_1=(C_3,\sigma_1,\mu_1)$ and $\Gamma_2=(P_2,\sigma_2,\mu_2)$ be a signed cycle on $3$-nodes and (1,1)-co-regular signed graph on $2$-nodes respectively(see Figure \ref{f1}). The spectrum of $\Gamma_1$ and $\Gamma_2$ are $S(A(\Gamma_1))=\{-1^{(2)},2\}$ and $S(A(\Gamma_2))=\{-1,1\}$. The spectrum of neighbourhood corona product $\Gamma_1*\Gamma_2$(see Figure \ref{f1}) can be calculated using Proposition \ref{Prop1} as follows
    \begin{enumerate}[label=(\roman*)]
        \item For $-1\in S(A(\Gamma_1))$ with multiplicity 2, we have $\bigg(\frac{(-1)+1\pm\sqrt{(-1-1)^2+4\cdot2(-1)^2}}{2}\bigg)^2=\big(\sqrt{3}^{(2)},-\sqrt{3}^{(2)}\big)$ and for $2\in S(A(\Gamma_1))$, we have $\frac{2+1\pm\sqrt{(2-1)^2+4\cdot2(2)^2}}{2}=(-1.3723,4.3723)$.
        \item The eigenvalue $-1$ with multiplicity $3$.
    \end{enumerate}
    Thus the A-spectrum of $\Gamma_1*\Gamma_2$ is $S(A(\Gamma_1*\Gamma_2))=\{\pm\sqrt{3}^{(2)},-1.3723,4.3723,(-1)^{(3)}\}$. \vspace{.3cm}

    \noindent Next, we calculate the spectrum of $Q(\Gamma_1*\Gamma_2)$. The spectrum of $Q(\Gamma_1)$ and $Q(\Gamma_2)$ are $S(Q(\Gamma_1))=\{1^{(2)},4\}$ and $S(Q(\Gamma_2))=\{0,2\}$ respectively. The spectrum of $Q(\Gamma_1*\Gamma_2)$ can be calculated using Proposition \ref{Prop 3} as follows
    \begin{enumerate}[label=(\roman*)]
        \item For $4\in S(Q(\Gamma_1))$, we have $\frac{(2+1+1+4+4)\pm\sqrt{(2+1+1-4-4)^2+32}}{2}=(2.5359,9.4641)$ and for $1\in S(Q(\Gamma_2))$ with multiplicity 2 we have $\bigg(\frac{(2+1+1+4+1)\pm\sqrt{1+8}}{2}\bigg)^2=(3^{(2)},6^{(2)})$
        \item The eigenvalue $2$ with multiplicity $3$.
    \end{enumerate}
    \noindent Thus the Q-spectrum of $\Gamma_1*\Gamma_2$ is $S(Q(\Gamma_1*\Gamma_2))=\{2^{(3)},2.5359,3^{(2)},6^{(2)},9.4641\}$.\vspace{.3cm}
    
    \noindent Lastly, we calculate the spectrum of $L(\Gamma_1*\Gamma_2)$. The spectrum of $L(\Gamma_1)$ and $L(\Gamma_2)$ are $S(L(\Gamma_1))=\{0,3^{(2)}\}$ and $S(L(\Gamma_2))=\{0,2\}$ respectively. The spectrum of $L(\Gamma_1*\Gamma_2)$ can be calculated using Proposition \ref{Prop 5} as follows
    \begin{enumerate}[label=(\roman*)]
        \item For $0\in S(L(\Gamma_1))$, we have $\frac{(2+1-1+4)\pm\sqrt{(2+1-1-4)^2+32}}{2}=(0,6)$ and for $3\in\sigma(L(\Gamma_2))$ with multiplicity 2 we have $\bigg(\frac{(2+1-1+4+3)\pm\sqrt{25+8}}{2}\bigg)^2=\big(1.6277^{(2)},7.3723^{(2)}\big)$
        \item The eigenvalue $4$ with multiplicity $3$.
    \end{enumerate}
    \noindent Thus the L-spectrum of $\Gamma_1*\Gamma_2$ is $S(L(\Gamma_1*\Gamma_2))=\{0,1.6277^{(2)},4^{(3)},6,7.3723^{(2)}\}$. \vspace{.5cm}
\end{example}

    \section*{Acknowledgement}
    We would like to acknowledge National Institute of Technology Sikkim for giving doctoral fellowship to Bishal Sonar.

\bibliographystyle{abbrv}
\bibliography{main.bib}

\end{document}